\documentclass[12pt]{amsart}

\usepackage{tikz}
\usepackage{amsmath}
\usepackage{amssymb}
\usepackage{xspace}
\usepackage{latexsym}
\usepackage{epsfig}
\usepackage{verbatim}
\usepackage{texdraw}

\newtheorem{theorem}{Theorem}[section]
\newtheorem{lemma}[theorem]{Lemma}
\newtheorem{sublemma}{Sublemma}[theorem]
\newtheorem{corollary}[theorem]{Corollary}

\numberwithin{equation}{section}

\newcommand{\ba}{\backslash}
\newcommand{\co}{\text{co}}
\newcommand{\si}{\text{si}}
\newcommand{\cl}{\textrm{cl}}
\newcommand{\ra}{\text{r}}

\newcommand{\la}{\lambda}

\begin{document}
\title[On contracting hyperplane elements.]
{On contracting hyperplane elements from a 3-connected matroid.}
\author{Rhiannon Hall}
\address{Brunel University, Uxbridge UB8 3PH, United Kingdom}
\email{rhiannon.hall@brunel.ac.uk}
\thanks{This research was supported by a Nuffield Foundation Award
for Newly Appointed Lecturers in Science, Engineering and
Mathematics}
\date{August 2007}
\subjclass{05B35} \keywords{matroid, 3-connected, hyperplane,
minor, chain theorem}
\begin{abstract}
Let $\tilde{K}_{3,n}$, $n\geq 3$, be the simple graph obtained
from $K_{3,n}$ by adding three edges to a vertex part of size
three. We prove that if $H$ is a hyperplane of a 3-connected
matroid $M$ and $M \not\cong M^*(\tilde{K}_{3,n})$, then there is
an element $x$ in $H$ such that the simple matroid associated with
$M/x$ is 3-connected.
\end{abstract}
\maketitle

\section{Introduction.}
Much work has been done recently on chain-type theorems and
splitter-type theorems for 3-connected matroids. In trying to
solve major problems, it is believed that we will need many of
these ``tools'' in order to make further progress. Our paper
contributes to this recent work by establishing the existence of
removable elements in particular structures of matroids. Now,
3-connectivity plays a major role in such work, as many
difficulties arise when working with matroids having
2-separations. However, considering only 3-connected matroids does
not restrict the power of our results. We rarely lose generality
by just considering 3-connected matroids, since all matroids can
be constructed from sums and 2-sums of 3-connected matroids.

When considering 3-connected matroids and their minors, holding on
to 3-connectivity can be a formidable task. We often wish to
understand what structures the original matroid would possess if
we lose 3-connectivity through taking a minor of it. Further, in
many applications, the presence of series or parallel classes does
not create any of the major difficulties that more substantial
2-separations can cause. Thus, there is also a lot of interest in
understanding the structures of 3-connected matroids from which we
take a minor, and the simplification or cosimplification of that
minor is not 3-connected.

A common problem arising from numerous contexts, is that we are
given a 3-connected matroid $M$ with sets of elements $X$ and $Y$
such that we wish to contract some element of $X$ or delete some
element of $Y$. What structure might $M$ have if for all $x\in X$
and for all $y\in Y$, $\si(M/x)$ and $\co(M\ba y)$ (the
simplification of $M/x$ and cosimplification of $M\ba y$) are not
3-connected? Oxley et. al.~\cite{basis} considered this problem in
the case where $X$ is a basis and $Y$ is its corresponding
cobasis, in fact they also proved a stronger result, a
splitter-type theorem. Our paper focuses on the case where $X$ is
a hyperplane, a problem suggested by Whittle in a private
communication. In other words, which 3-connected matroids $M$ have
the property that they contain a hyperplane $H$, such that for all
$h\in H$, $\si(M/h)$ is not 3-connected? Let $\tilde{K}_{3,n}$,
$n\geq 3$, be the simple graph obtained from $K_{3,n}$ by the
addition of three edges to a vertex part of size three. We will
show that the set of all 3-connected matroids having the
hyperplane contraction property just stated, is the family
$\mathcal{P}^*$, which we define to be the family of all cographic
matroids  $M^*(\tilde{K}_{3,n})$, $n\geq 3$. Note that if we
define $\mathcal{P}$ to be the family of all matroids whose dual
is a member of $\mathcal{P}^*$, then $\mathcal{P}$ is equal to the
family of all matroids $M$, such that $E(M)=\{c_1,c_2,c_3,
t_{11},t_{12},t_{13}, t_{21},t_{22},t_{23}, \dots,
t_{n1},t_{n2},t_{n3}\}$, $n\geq 3$, where $\{c_1,c_2,c_3\}$ is a
triangle, $\{t_{i1},t_{i2},t_{i3}\}$ is a triad, and where
$M|_{\{c_1,c_2,c_3,t_{i1},t_{i2},t_{i3}\}}\cong K_4$ for all
$1\leq i\leq n$. This equivalence can be seen by observing that
these matroids are all graphic, as none of them have a minor
isomorphic to $U_{2,4}$, $F_7$, $F_7^*$, $M^*(K_5)$ or
$M^*(K_{3,3})$ (the five excluded minors for graphic matroids
found by Tutte~\cite{tutte}, see also Oxley~\cite[Theorem
6.6.5]{oxley}). Diagrams of graphic and geometric representations
of the matroids of $\mathcal{P}$ are given in
Figure~\ref{dualpics}.

\begin{figure}
\begin{tikzpicture}
    \draw (-1,2)--(1,2);
    \draw (-1,2) .. controls (0,2.8) .. (1,2);
    \draw (-1,2)--(-2,0);
    \draw (-1,2)--(-1,0);
    \draw (-1,2)--(0,0);
    \draw (-1,2)--(2.5,0);
    \draw (0,2)--(-2,0);
    \draw (0,2)--(-1,0);
    \draw (0,2)--(0,0);
    \draw (0,2)--(2.5,0);
    \draw (1,2)--(-2,0);
    \draw (1,2)--(-1,0);
    \draw (1,2)--(0,0);
    \draw (1,2)--(2.5,0);
    \filldraw[black] (0,0)circle(3pt) (-1,0)circle(3pt)
    (-2,0)circle(3pt) (2.5,0)circle(3pt) (-1,2)circle(3pt)
    (0,2)circle(3pt) (1,2)circle(3pt);
    \draw (1,0)node{\textbf{\dots}};
    \draw (-2,-0.4)node{1} (-1,-0.4)node{2} (0,-0.4)node{3} (2.5,-0.4)node{$n$};
    \draw (3.9,-1.5)node{\phantom{d}};
\end{tikzpicture}
\begin{tikzpicture}
    \draw (0,-1.3)--(-1.725,-0.15)--(-1.725,2.45)--(0,1.3)--(0,-1.3);
    \draw (0,1)--(-1.5,0.5);
    \draw (0,0)--(-1.5,0.5);
    \draw (0,-1)--(-1,0.67);
    \filldraw[black] (0,1)circle(3pt) (0,0)circle(3pt) (0,-1)circle(3pt)
    (-0.75,0.25)circle(3pt) (-1.5,0.5)circle(3pt) (-1,0.67)circle(3pt);

    \draw (0,-1.3)--(-2.3,-1.3)--(-2.3,1.3)--(-1.725,1.3);

    \draw (0,-1.3)--(-1.84,-2.45)--(-1.84,-1.3);

    \draw (0,-1.3)--(-1.38,-3.6)--(-1.38,-2.16);

    \draw (0,-1.3)--(1.725,1)--(1.725,3.6)--(0,1.3)--(0,-1.3);
    \draw (0,1)--(1.5,1.5);
    \draw (0,0)--(1.5,1.5);
    \draw (0,-1)--(1,1.33);
    \filldraw[black] (0,1)circle(3pt) (0,0)circle(3pt) (0,-1)circle(3pt)
    (0.75,0.75)circle(3pt) (1.5,1.5)circle(3pt) (1,1.33)circle(3pt);

    \draw (0,-1.3)--(2.3,-3.6)--(2.3,-1)--(1.114,0.186);
    \draw (0.746,-0.305)--(2,-2.5);
    \draw (0.503,-0.629)--(2,-2.5);
    \draw (0.189,-1.048)--(1.33,-1.33);
    \filldraw[black] (0,1)circle(3pt) (0,0)circle(3pt) (0,-1)circle(3pt)
    (1,-1.25)circle(3pt) (2,-2.5)circle(3pt) (1.33,-1.33)circle(3pt);

    \draw[gray] (0,1.3)--(1.114,0.186);
    \draw[gray] (0.746,-0.305)--(0,1);
    \draw[gray] (0.503,-0.629)--(0,0);
    \draw[gray] (0.189,-1.048)--(0,-1);
    \draw[thick] (0,-1.3)--(0,1.3);

    \draw[gray] (0,1.3)--(-1.725,1.3);
\end{tikzpicture}
\caption{A graphic representation and a geometric representation
of matroid from the class $\mathcal{P}$, the class of matroids
having a cohyperplane $H$ such that for all $x\in H$, $\co(M\ba
x)$ is not 3-connected. Within the graph, $H$ consists of all of
the edges of the $K_{3,n}$ subgraph. Within the geometric
representation, $H$ consists of all elements that are not in the
three-point line that is common to all copies of $K_4$.}
\label{dualpics}
\end{figure}
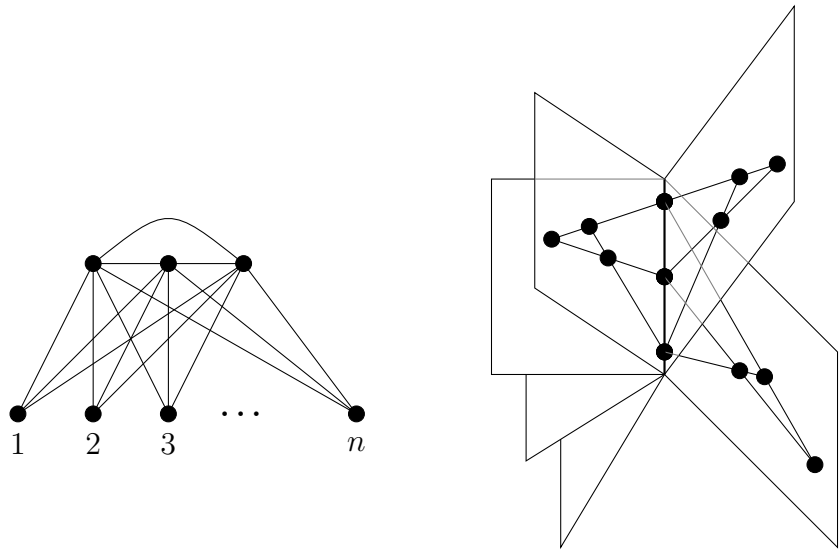

Stated formally, the main theorem of this paper is as follows.
\begin{theorem}\label{main}
Let $M$ be a 3-connected matroid. Then $M$ has a hyperplane $H$
such that for all $h\in H$, $\si(M/h)$ is not 3-connected if and
only if $M\cong M^*(\tilde{K}_{3,n})$ for some $n\geq 3$.
\end{theorem}

Note that within the matroid $M=M^*(\tilde{K}_{3,n})$, the
hyperplane $H$ in question, consists of the elements that
correspond to the edges of the original $K_{3,n}$ graph.
Furthermore, if $h\in H$ then $\si(M/h)$ is not 3-connected due to
the existence of a single series pair, and $\co(\si(M/h))$ is
3-connected with $\co(\si(M/h))\cong M^*(\tilde{K}_{3,n-1})$.

In proving Theorem~\ref{main}, we also prove a lemma that will be
of independent interest, as it may be applicable to various
situations where we have a set of elements from which we wish to
contract some member and keep a particular 3-connected minor. This
lemma is:

\begin{theorem}\label{main2}
Let $(X_1,x,X_2)$ be a vertical 3-partition of a 3-connected
matroid $M$. Then there exists $y\in X_i$, $i=1,2$, such that
$\si(M/y)$ is 3-coonnected.
\end{theorem}

The paper is structured as follows. Section 2 provides some
preliminary lemmas on matroid connectivity. In Section 3, we prove
Theorem~\ref{main2} as well as some lemmas specific to our
hyperplane problem. Section 4 completes the proof of
Theorem~\ref{main}. All terminology is taken from
Oxley~\cite{oxley}, with the exception that $\si(M)$ and $\co(M)$
denote the simplification and cosimplification of $M$
respectively.

\section{Preliminaries.}
This section will provide definitions and results, mostly on
connectivity, that are useful tools when applied to problems in
matroid structure theory. We begin the section with some
definitions on matroid connectivity. Let $M$ be a matroid on the
groundset $E(M)$. The function defined on all subsets of $E(M)$,
given by $\lambda(A)=\ra(A)+\ra(E(M)-A)-\ra(M)$, is known as the
\textit{connectivity function} of $M$. We say that a subset
$A\subseteq E(M)$ is \textit{$k$-separating} or a
\textit{$k$-separator of} $M$ if $\lambda(A)\leq k-1$, and we say
that a partition $(A,E(M)-A)$ is a \textit{$k$-separation} of $M$
if $\lambda(A)\leq k-1$ and $|A|,|E(M)-A|\geq k$. A $k$-separator
or $k$-separation is \textit{exact} if $\lambda(A)=k-1$. A matroid
$M$ is said to be \textit{$k$-connected} if $M$ has no
$k'$-separation for any $k'<k$. We define a \textit{$k$-partition}
of $M$ to be a partition $(A_1,A_2,\dots,A_n)$ of $E(M)$ in which
$A_i$ is $k$-separating for all $1\leq i\leq n$, and an
\textit{exact $k$-partition} is where $A_i$ is exactly
$k$-separating for all $1\leq i\leq n$.

It is well known that the connectivity function of a matroid is
submodular, that is for all $X,Y\subseteq E(M)$, we have
$\lambda(X\cap Y)+\lambda(X\cup Y)\leq\lambda(X)+\lambda(Y)$. From
this the following result of Geelen and Whittle~\cite{geewhi} is
elementary and will be used repeatedly. It is commonly referred to
as ``uncrossing''.
\begin{lemma}\label{uncrossing}
Let $M$ be a 3-connected matroid with 3-separating sets $X$ and
$Y$. Then
\begin{itemize}
\item if $|X\cap Y|\geq 2$ then $X\cup Y$ is a 3-separator of $M$;
\item if $|E(M)-(X\cup Y)|\geq 2$ then $X\cap Y$ is a 3-separator
of $M$;
\item if $|X\cap Y|=1$ then $X\cup Y$ is a 4-separator of $M$.
\end{itemize}
\end{lemma}
The following is straight forward and can be easily proved by
considering the dual matroid.
\begin{lemma}\label{circuit}
Let $X$ be a series class of a 2-connected matroid $M$ with
$y\in\cl\, (X)-X$. Then $X\cup\{y\}$ is a circuit of $M$.
\end{lemma}

We now define segments, cosegments and fans. These structures have
appeared often in the literature due to their high quantities of
triangles and triads. A \textit{segment} of a 3-connected matroid
is a set of elements in which every three-element subset is a
triangle. A \textit{cosegment} of a 3-connected matroid is a set
of elements in which every three element subset is a triad.
Segments and cosegments are known in some literature as lines and
colines respectively. A \textit{fan} of a 3-connected matroid is
an ordered set of elements $\{f_1,f_2,\dots,f_n\}$ in which
$\{f_i,f_{i+1},f_{i+2}\}$ is a triangle or a triad for all $1\leq
i\leq n-2$, where if $\{f_i,f_{i+1},f_{i+2}\}$ is a triangle then
$\{f_{i+1},f_{i+2},f_{i+3}\}$ is a triad, and if
$\{f_i,f_{i+1},f_{i+2}\}$ is a triad then
$\{f_{i+1},f_{i+2},f_{i+3}\}$ is a triangle.

The next three lemmas appear in \cite{sequences}, and are useful
when considering elements of a matroid that can be moved from one
side of a $k$-separation to the other.
\begin{lemma}\label{coclosure}
Let $M$ be a matroid and let $(X,Y,\{z\})$ be a partition of
$E(M)$. Then $z\in\cl^{\, *}(Y)$ if and only if $z\notin\cl\,
(X)$.
\end{lemma}
\begin{proof}
By the well known formula for corank, see for example
Oxley~\cite{oxley}, we have $\ra^*(Y)=|Y|+\ra(X\cup\{z\})-\ra(M)$,
and $\ra^*(Y\cup\{z\})=|Y\cup\{z\}|+\ra(X)-\ra(M)$. It follows
easily that $z\in\cl^*(Y)$ if and only if
$\ra(X\cup\{z\})=\ra(X)+1$ if and only if $z\notin\cl(X)$.
\end{proof}

\begin{lemma}\label{guts}
Let $(X,Y,\{z\})$ be a partition of $E(M)$ for some matroid $M$.
If $\lambda(X)=\la(Y)$ then either $z\in\cl\, (X)\cap\cl\, (Y)$ or
$z\in\cl^{\, *}(X)\cap\cl^{\, *}(Y)$.
\end{lemma}
\begin{proof}
Firstly, $\la(X)=\ra(X)+\ra(Y\cup\{z\})-\ra(M)$ and
$\la(Y)=\ra(X\cup\{z\})+\ra(Y)-\ra(M)$. Since $\la(X)=\la(Y)$, it
follows that $\ra(Y\cup\{z\})-\ra(Y)=\ra(X\cup \{z\})-\ra(X)$, and
we see that $z\in\cl(X)$ if and only if $z\in\cl(Y)$. The result
now follows from Lemma \ref{coclosure}.
\end{proof}

\begin{lemma}\label{cosegment}
Let $M$ be a 3-connected matroid with a 3-separating set $A$. For
some $n\geq 3$, let $\{x_1,x_2,\dots,x_n\}\subseteq E(M)-A$. If
$x_i\in\cl\, (A)$ for all $1\leq i\leq n$, then
$\{x_1,x_2,\dots,x_n\}$ is a segment of $M$. Dually, if
$x_i\in\cl^{\, *} (A)$ for all $1\leq i\leq n$, then
$\{x_1,x_2,\dots,x_n\}$ is a cosegment of $M$.
\end{lemma}
\begin{proof}
We prove only the first part, as the second part follows by
duality. Suppose that $x_i\in\cl\, (A)$ for all $1\leq i\leq n$.
Then $\ra(A\cup\{x_1,x_2,\dots,x_n\})=\ra(A)$. By the
submodularity of the rank function, we have
$\ra(A\cup\{x_1,x_2,\dots,x_n\})+\ra(E(M)-A)\geq \ra
((A\cup\{x_1,x_2,\dots,x_n\})\cap
(E(M)-A))+\ra((A\cup\{x_1,x_2,\dots,x_n\})\cup (E(M)-A))$, which
implies that $\ra(A)+\ra(E(M)-A)\geq
\ra(\{x_1,x_2,\dots,x_n\})+\ra(M)$. Now,
$\lambda(A)=\ra(A)+\ra(E(M)-A)-\ra(M)=2$ since $A$ is
3-separating, giving $\ra(\{x_1,x_2,\dots,x_n\})\leq 2$. Since $M$
is 3-connected and $n\geq 3$, we conclude that
$\{x_1,x_2,\dots,x_n\}$ is a segment of $M$.
\end{proof}

The following is well known and straight forward.
\begin{lemma}\label{cosegment2}
Let $M$ be a 3-connected matroid with a cosegment $D$ such that
$|D|\geq 4$. Then for all $d\in D$, $M/d$ is 3-connected.
\end{lemma}
\begin{proof}
Suppose this is false, and let $(A,B)$ be a 2-separation of $M/d$.
We may assume without loss of generality that $|A\cap D|\geq 2$,
and hence $d\in\cl^*_{M}(A)$. Then $(A\cup\{d\},B)$ is a
2-separation of $M$, a contradiction.
\end{proof}

The next two results are well known, see for example \cite{basis}.
They provide us with information on why a matroid might lose some
level of connectivity upon the contraction of an element. The
second of these lemmas focuses on the case where we contract an
element from a 3-connected matroid and not only do we lose
3-connectivity in the resultant matroid, but we lose
3-connectivity in its simplification as well.
\begin{lemma}\label{contr1}
Let $M$ be a $k$-connected matroid of size $|E(M)|\geq 2k-1$, with
an element $z$ such that $M/z$ is not $k$-connected. Then $M/z$ is
$(k-1)$-connected with at least one $(k-1)$-separation. Moreover,
if $(X,Y)$ is a $(k-1)$-separation of $M/z$ then $(X,Y,\{z\})$ is
a partition of $E(M)$ such that $\la_M(X)=\la_M(Y)=k$, $z\in\cl\,
(X)\cap\cl\, (Y)$, and $|X|,|Y|\geq k-1$.
\end{lemma}

\begin{lemma}\label{contr2}
Let $M$ be a 3-connected matroid with an element $z$ such that
$\si(M/z)$ is not 3-connected. Then $M$ has a partition
$(X,Y,\{z\})$ such that $\la(X)=\la(Y)=2$, $z\in\cl\, (X)\cap
\cl\, (Y)$ and $\ra(X),\ra(Y)\geq 3$.
\end{lemma}
\begin{proof}
Firstly, since $\si(M/z)$ is not 3-connected, $\si(M/z)$ has a
2-separation $(X',Y')$ with $\ra_{M/z}(X'),\ra_{M/z}(Y')\geq 2$.
It follows that $M/z$ has a 2-separation $(X,Y)$ with
$\ra_{M/z}(X),\ra_{M/z}(Y)\geq 2$. Hence, we must also have
$\ra(\si(M/z))\geq 3$, implying that $\ra(M)\geq 4$ and thus
$|E(M)|\geq 6$. By Lemma~\ref{contr1}, $\la_M(X)=\la_M(Y)=2$ and
$z\in\cl(X)\cap \cl(Y)$. Hence $\ra_M(X)=\ra_{M/z}(X)+1$ and
$\ra_M(Y)=\ra_{M/z}(Y)+1$, and the result follows.
\end{proof}

We refer to a partition $(X,Y,\{z\})$ in which
$\la(X)=\la(Y)=k-1$, $z\in\cl(X)\cap \cl(Y)$, and
$\ra(X),\ra(Y)\geq k$ as a \textit{vertical $k$-partition}. Note
how this differs slightly from a \textit{vertical $k$-separation},
which is defined in many papers as a $k$-separation $(A,B)$ in
which $\ra(A),\ra(B)\geq k$, see for example \cite{basis}. The
following is a useful tool when considering vertical
$k$-partitions.
\begin{lemma}\label{vertcl}
Let $M$ be a $k$-connected matroid with a vertical $k$-partition
$(X,Y,\{z\})$. Then $(X-\cl\, (Y),\cl\, (Y)-\{z\},\{z\})$ is also
a vertical $k$-partition of $M$.
\end{lemma}
\begin{proof}
Suppose that $x\in X\cap \cl(Y)$. Then
$\la(X-\{x\})\in\{k-2,k-1\}$. If $\la(X-\{x\})=k-2$ then
$(X-\{x\},Y\cup \{x,z\})$ is a $(k-1)$-separation of $M$,
contradicting that $M$ is $k$-connected. Hence $\la(X-\{x\})=k-1$
which implies that $\ra(X-\{x\})=\ra(X)$ and $z\in\cl(X-\{x\})$.
It follows that $(X-\{x\},Y\cup\{x\},z)$ is a vertical
$k$-partition of $M$. Continuing this process, we see that
$(X-\cl(Y),\cl(Y)-\{z\},\{z\})$ is a vertical $k$-partition of
$M$.
\end{proof}

We now state the version of Bixby's Theorem~\cite{Bixby} that is
most natural for our requirements in this paper.
\begin{theorem}[Bixby's theorem] \label{bixby}
Let $M$ be a 3-connected matroid and let $x\in E(M)$. Then either
$M\ba x$ is 3-connected up to series pairs or $M/x$ is 3-connected
up to parallel pairs.
\end{theorem}

We now discuss the important concept of local connectivity. The
\textit{local connectivity function} of a matroid $M$ is defined
on pairs of subsets of $E(M)$ as
$\sqcap(A,B)=\ra(A)+\ra(B)-\ra(A\cup B)$. Note that we do not
require $A$ and $B$ to be disjoint. It is helpful to think of
local connectivity as the connectivity between $A$ and $B$ in the
matroid $M|_{(A\cup B)}$. A good introduction to the local
connectivity function can be found in Oxley, Semple, \& Whittle
\cite{flowers}. The following two results on local connectivity
appear in \cite{flowers}.

\begin{lemma}\label{flowers1}
Let $(X,Y,Z)$ be an exact 3-partition of the 3-connected matroid
$M$. Then $\sqcap(X,Y)=\sqcap(X,Z)=\sqcap(Y,Z)$.
\end{lemma}
\begin{lemma}\label{flowers2-}
Let $X$ and $Y$ be subsets of $E(M)$, with $X'\subseteq X$ and
$Y'\subseteq Y$. Then $\sqcap(X',Y')\leq \sqcap(X,Y)$.
\end{lemma}

At this point, we use Lemma~\ref{flowers2-} to prove the
following.
\begin{lemma}\label{flowers2}
Let $X$ and $Y$ be disjoint subsets of $E(M)$, with $X'\subseteq
X$, and such that $\sqcap(X',Y)=\sqcap(X,Y)$. If there exists
$y\in Y$ such that $y\in\cl\, (X)$, then $y\in\cl\, (X')$.
\end{lemma}
\begin{proof}
Firstly, since $y\in\cl(X)$, we have
\begin{align*} \sqcap(X\cup\{y\},
Y)&=\ra(X\cup\{y\})+\ra(Y)-\ra(X\cup\{y\}\cup Y)\\
&=\ra(X)+\ra(Y)-\ra(X\cup Y)\\
&=\sqcap(X,Y)\\
&=\sqcap(X',Y).
\end{align*}
By Lemma~\ref{flowers2-}, we have $\sqcap(X'\cup\{y\}, Y)\leq
\sqcap(X\cup\{y\},Y)=\sqcap(X',Y)$. Lemma~\ref{flowers2-} also
gives $\sqcap(X',Y)\leq \sqcap(X'\cup\{y\},Y)$, therefore we may
deduce that $\sqcap(X'\cup\{y\},Y)=\sqcap(X',Y)$. It follows that
$\ra(X'\cup\{y\})+\ra(Y)-\ra(X'\cup\{y\}\cup
Y)=\ra(X')+\ra(Y)-\ra(X'\cup Y)$, and by cancelling terms, we
obtain $\ra(X'\cup \{y\})=\ra(X')$. The result now follows.
\end{proof}

\section{Some useful lemmas.}
The purpose of this section is to prove some ``larger'' lemmas
including Theorem~\ref{main2}. Most of the lemmas of this section
are specific to our problem, however Theorem~\ref{main2} and
Lemma~\ref{segment-1} may be applicable to a number of far more
general settings.

We begin by generating a lower bound on the size of a matroid that
has a hyperplane from which contraction of any element creates a
vertical 2-separation.

\begin{lemma}\label{size}
Let $M$ be a 3-connected matroid with a hyperplane $H$, such that
for all $h\in H$, $si(M/h)$ is not 3-connected. Then $|E(M)|\geq
7$.
\end{lemma}
\begin{proof}
Let $M$ be such a matroid. By Lemma \ref{contr2}, $M$ has a
vertical 3-partition $(X,Y,\{z\})$. Since $\ra(X),\ra(Y)\geq 3$,
we have $|X|,|Y|\geq 3$. Hence $|E(M)|\geq 3+3+1=7$.
\end{proof}

We next consider what happens if our matroid has a specific type
of 3-separator.
\begin{lemma}\label{segment-1}
Let $M$ be a 3-connected matroid with a 3-separation $(A,B)$ such
that $\ra (A)=3$, and there exists $e$ with $\ra (A-\{e\})=2$ and
$|A-\{e\}-\cl\; (B)|\geq 3$. See Figure~\ref{segpic} for a
geometrical representation of $(A,B)$. Suppose that $\si(M/a)$ is
not 3-connected for all $a\in A-\{e\}$. Then $M$ has a cosegment
$D$ such that $e\in D$, $|D-\{e\}|\geq |A-\{e\}-\cl\; (B)|$, and
$A\cup D$ is 3-separating in $M$.
\end{lemma}

\begin{figure}
\begin{tikzpicture}
    \draw (0,0)--(3.47,0)--(3.47,2.2)--(0,2.2)--(0,0);
    \draw (-0.6,0)--(4.07,0);
    \draw [rotate=30] (0,0) arc (110:330:2.5cm and 1.2cm);
    \draw (0.5,0.6)--(3,1.9);
    \filldraw[black] (2.4,0.7)circle(3pt)
    (2.5,1.64)circle(3pt) (2,1.38)circle(3pt) (1.5,1.12)circle(3pt);
    \draw (1,0.86)circle(3pt);
    \draw (2.7,0.7)node{$e$} (2.3,1.9)node{$x$}
    (1.8,1.6)node{$y$} (1.3,1.4)node{$z$} (4.8,-1.1)node{$B$}
    (4.8,1.1)node{$A$}; 
\end{tikzpicture}
\caption{A 3-separation $(A,B)$, such that $\ra(A)=3$ and
$A-\{e\}-\cl(B)$ is a segment with at least three elements.}
\label{segpic}
\end{figure}
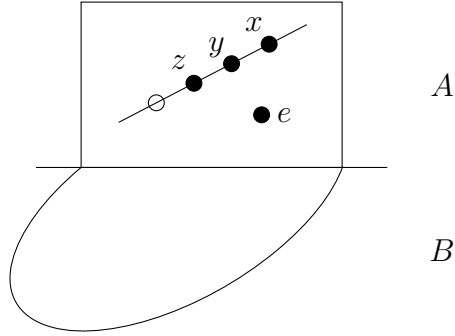

\begin{proof}
Firstly, we may assume that $\cl(A-\{e\})\subseteq A$ (we will use
this assumption in the proof of Sublemma~\ref{sub4}). Now, $M\ba
e$ has the vertical 2-separation $(A-e,B)$, so
$e\in\cl^*(A-\{e\})\cap \cl^*(B)$. Since $\si(M/a)$ is not
3-connected for any $a\in A-\{e\}$, upon the contraction of any
member of $A-\{e\}$, we will obtain a vertical 2-separation, which
corresponds to a vertical 3-partition of the original matroid $M$.
Let $x,y\in A-\{e\}-\cl(B)$, and consider their vertical
3-partitions. Let $(X_1,x,X_2)$ be a vertical 3-partition. Then,
since $x\in\cl(X_i)$, $i=1,2$, and $x\notin\cl(B)$, we see that
$X_i\cap (A-\cl(B))\ne\emptyset$. Assume without loss of
generality that $e\in X_2$. Then $\ra(A\cap X_1)\leq 2$, and as
$\ra(X_1)\geq 3$ we have $X_1\cap B\ne \emptyset$.

\begin{sublemma}\label{sub1}
$X_2\cap A=\{e\}$ and $A\not\subseteq \cl\, (X_i)$, $i=1,2$.
\end{sublemma}

\begin{proof}
Suppose that $A\subseteq\cl(X_2)$. It follows that if $|X_1-
A|\geq 2$ then $X_1-A$ is 2-separating in $M$, and if $|X_1- A|=1$
then $X_1-A$ is separating in $M$, and both possibilities
contradict the 3-connectivity of $M$. We conclude that
$A\not\subseteq \cl\, (X_2)$.

It now follows that since $e\in X_2$ and $x\in\cl(X_2)$, no member
of $A-\{e,x\}$ can be in $X_2$, otherwise $A$ would be contained
in $\cl(X_2)$. Thus $X_2\cap A=\{e\}$.

Now suppose that $e\in\cl(X_1)$. Then as $X_2\cap A=\{e\}$, we
have $e\notin\cl(X_2-\{e\})$, implying that $X_2-\{e\}$ is a
2-separator of $M$ of size at least two, a contradiction. Thus
$e\notin\cl(X_1)$. Therefore $A\not\subseteq\cl(X_1)$ as required.
\end{proof}

\begin{sublemma}\label{sub2}
$\sqcap(X_2-A,A)=\sqcap(X_2-A,\{e,x\})=1$ and
$\sqcap(X_1-A,A)=\sqcap(X_1-A,A-\{e\})=\begin{cases} 0 &\text{if
$|X_1-A|=1$;}\\ 1 &\text{if $|X_1-A|\geq 2$,}\end{cases}$.
\end{sublemma}

\begin{proof}
Since $X_i-\cl(B)\ne\emptyset$ and $X_i-A\subseteq B$, with
$\sqcap(B,A)=2$ and $A\not\subseteq\cl(X_i)$, we see that
$\sqcap(X_i-A,A)\leq 1$, $i=1,2$. It is easily seen that since
$x\in\cl(X_2)$ with $X_2\cap A=\{e\}$, we must have
$\sqcap(X_2-A,\{e,x\})=\sqcap(X_2-A,A)=1$.

Now, if $|X_1-A|\geq 2$ then $\lambda(X_1-A)\geq 2$. Since
$\lambda(X_1\cup\{x\})=2$ and $\ra(X_2\cup A)=\ra(X_2)+1$, we must
have $\sqcap(X_1-A,A-\{e\})=1$, otherwise we would have
$\lambda(X_1-A)=1$. Now, by Lemma~\ref{flowers1},  $\sqcap
(X_1-A,A)=1$.

If $|X_1-A|=1$, then since $\ra (X_1)\geq 3$ and
$A\not\subseteq\cl(X_1)$, we must have $X_1-A\subseteq
(B-\cl(A))$, thus $\sqcap(X_1-A,A-\{e\})=\sqcap(X_1-A,A)=0$.
\end{proof}

\begin{sublemma}\label{sub3}
$X_1-A$ and $X_2-A$ are 3-separators of $M$.
\end{sublemma}
\begin{proof}
This follows immediately from the fact that $\ra(X_i\cup
A)=\ra(X_i)+1$ and $\ra(X_i-A)<\ra(X_i)$, $i=1,2$.
\end{proof}

Let $(Y_1,y,Y_2)$ be a vertical 3-partition of $M$, where $e\in
Y_2$. By symmetry, the same conditions as described in
Sublemmas~\ref{sub1}--\ref{sub3} for $(X_1,x,X_2)$ also hold for
$(Y_1,y,Y_2)$. Let $X_i'=X_i\cap B$ and $Y_i'=Y_i\cap B$, $i=1,2$.

\begin{sublemma}\label{sub4}
Either $|X_1'|=1$ or $|Y_1'|=1$.
\end{sublemma}

\begin{proof}
Suppose that $|X_i'|\geq 2$ and $|Y_i'|\geq 2$, $i=1,2$. Then
$\sqcap(\{x,y\},X_1')=1$, $\sqcap(\{x,y\}, Y_1')=1$, $\sqcap
(\{e,x\},X_2')=1$ and $\sqcap(\{e,y\},Y_2')=1$ from
Sublemma~\ref{sub2}. We consider how these sets can intersect. It
is evident that $X_2'$ contains some element that is not a member
of $Y_2'$ since $x\notin\cl(Y_2'\cup \{e\})$ and $x\in\cl(X_2'\cup
\{e\})$. Similarly, $Y_2'$ contains some element that is not a
member of $X_2'$. Furthermore, since $e\in\cl(X_2'\cup\{x\})$ but
$e\notin\cl(Y_1'\cup\{x\})=\cl(Y_1)$, we see that $X_2'$ contains
some element not in $Y_1'$. Finally, as $x\in\cl(X_1'\cup\{y\})$
and $x\notin\cl(Y_2'\cup\{y\})$, there must exist some element of
$X_1'$ that is not in $Y_2'$. The result of this is that each of
the sets $X_1'\cap Y_1'$, $X_1'\cap Y_2'$, $X_2'\cap Y_1'$, and
$X_2'\cap Y_2'$ are nonempty. Now, since $e\in X_2\cap Y_2$ and
$X_2'\cap Y_2'\ne\emptyset$, it follows that $|X_2\cap Y_2|\geq
2$, so by uncrossing $X_2\cup Y_2$ is 3-separating in $M$.

Now observe that no member of $A-\{e\}$ is in $X_2\cup Y_2$, but
every member of $A-\{e\}$ is in $\cl(X_2\cup Y_2)$ since
$\{x,y\}\subseteq\cl(X_2\cup Y_2)$. Therefore $\ra(X_2\cup Y_2\cup
A)=\ra(X_2\cup Y_2)$. However, $E(M)-(X_2\cup Y_2\cup A)=X_1'\cap
Y_1'$. Suppose that $|X_1'\cap Y_1'|\geq 2$ then $\ra(X_1'\cap
Y_1')\leq \ra((X_1\cap Y_1)\cup\{x,y\})-1$ implying that
$\lambda(X_1'\cap Y_1')\leq 1$ (because $\lambda(X_2\cup Y_2\cup
A)\leq\lambda(X_2\cup Y_2)\leq 2$), contradicting the connectivity
of $M$. Therefore $|X_1'\cap Y_1'|=1$, but since
$\cl(A-\{e\})\subseteq A$, we have $\ra(X_1'\cap Y_1')\leq
\ra(X_1\cap Y_1)\cup\{x,y\})-2$ implying that $\lambda(X_1'\cap
Y_1')=0$, another contradiction to the connectivity of $M$. This
contradiction shows that it is not possible to have $|X_i'|\geq 2$
and $|Y_i'|\geq 2$, $i=1,2$.
\end{proof}

We may now assume by Sublemma~\ref{sub4} and by symmetry that
$|X_1'|=1$. Recall that $|X_2'|,|Y_2'|\geq 2$ because
$|X_2|,|Y_2|\geq 3$. First note that since $\ra(X_1)\geq 3$ and
$e\notin\cl(X_1)$, it follows that the element $x'$ of $X_1'$ is
not in $\cl(A)$. Then
$x'\in\cl^*(A-\{e\})\cap\cl^*(B\cup\{e\}-\{x'\})$ by
Lemma~\ref{guts}. Now, since $y\in\cl(Y_2'\cup\{e\})$ and
$y\notin\cl(X_2'\cup\{e\})$, it follows that $Y_2'$ must contain
$x'$, and hence $Y_1'\subseteq X_2'$.

\begin{sublemma}\label{sub5}
$|Y_1'|=1$.
\end{sublemma}

\begin{proof}
Suppose that $|Y_1'|\geq 2$. Then $\sqcap(Y_1',A-\{e\})=1$ by
Sublemma~\ref{sub2}, which implies that
$A-\{e,x\}\subseteq\cl(Y_1'\cup\{x\})$. However
$y\notin\cl(X_2'\cup\{x\})\supseteq\cl(Y_1'\cup\{x\})$, a
contradiction. This implies that $|Y_1'|=1$.
\end{proof}

A similar argument to the one above shows that the element $y'$ of
$Y_1'$ is a member of
$\cl^*(A-\{e\})\cap\cl^*(B\cup\{e\}-\{y'\})$. Also, since $y'\in
Y_2'$, we have $x'\ne y'$.

We may apply symmetric arguments to those above to any pair of
elements $x,y\in A-\{e\}-\cl(B)$, where $(X_1,x,X_2)$ and
$(Y_1,y,Y_2)$ are vertical 3-partitions of $M$ such that $e\in
X_2$ and $e\in Y_2$. These arguments show that
$|X_1-A|=|Y_1-A|=1$, and if $x'\in X_1-A$ and $y'\in Y_1-A$, then
$x'\ne y'$ and $x',y'\in\cl^*(A-\{e\})$. Now, let
$D=\cl^*(A-\{e\})-(A-\{e\})$. Then $D$ contains $e$ and
$\{x',y'\}$ for all $x,y\in A-\{e\}-\cl(B)$. Furthermore,
$|D-\{e\}|\geq |A-\{e\}-\cl(B)|$ since $x'\ne y'$ for all $x,y\in
A-\{e\}-\cl(B)$. We also see that $A\cup D$ is 3-separating by
construction, and that $D$ is a cosegment of $M$, by
Lemma~\ref{cosegment}.
\end{proof}

We may now apply Lemma~\ref{segment-1} to our problem in the
following corollary.

\begin{corollary}\label{segment}
Let $M$ be a 3-connected matroid with a 3-separation $(A,B)$ such
that $\ra\; (A)=3$, and there exists $e$ with $\ra\; (A-\{e\})=2$
and $|A-\{e\}-\cl(B)|\geq 3$. Again, refer to Figure~\ref{segpic}
for a geometrical representation of $(A,B)$. Suppose that $M$ has
a hyperplane $H$ that contains $A-\{e\}$. Then there exists $h\in
H$ such that $\si(M/h)$ is 3-connected.
\end{corollary}

\begin{proof}
Suppose we have a matroid satisfying such conditions, and suppose
that for all $a\in A-\{e\}$, $\si(M/a)$ is not 3-connected. Then
by Lemma~\ref{segment-1}, $e$ is a member of a cosegment $D$ of
size at least four, such that $A\cup D$ is 3-separating in $M$. In
order for $H$ to have a rank of $\ra(M)-1$, $H$ must intersect
$D$. Let $h\in H\cap D$. Then by Lemma~\ref{cosegment2}, $M/h$ is
3-connected.
\end{proof}

The following lemma allows us to choose vertical 3-partitions that
have a certain type of ``minimality'' on one of the large sides of
the partition.

\begin{lemma}\label{minimal} Let $M$
be a 3-connected matroid with a set of elements $J$, such that for
all $j\in J$, $\si(M/j)$ is not 3-connected. Suppose $x\in J$ and
$(X_1,x,X_2)$ is a vertical 3-partition such that for all $y\in
(X_1\cup\{x\})\cap J$, whenever $(Y_1,y,Y_2)$ is a vertical
3-partition with $Y_1\subseteq X_1$ then $Y_1\cap J \ne\emptyset$.
Then there exists $z\in (X_1\cup\{x\})\cap J$ with a vertical
3-partition $(Z_1,z,Z_2)$ such that
\begin{itemize}
\item $Z_1\subseteq X_1$ and $Z_1\cap J\ne\emptyset$, and
\item $Z_2\cup\{z\}$ is closed, and
\item for all $j\in Z_1\cap J$, whenever $(J_1,j,J_2)$ is a
 vertical 3-partition, then $J_1\cap X_2 \ne\emptyset$ and
 $J_2\cap X_2\ne\emptyset$.
\end{itemize}
\end{lemma}

\begin{proof}
In order to construct such a partition $(Z_1,z,Z_2)$, we begin by
checking the vertical 3-partition $(A_1,j_1,B_1)$, where
$A_1=X_1-\cl(X_2)$, $j_1=x$, and $B_1=\cl(X_2)-\{x\}$. Clearly,
$A_1\subseteq X_1$, $A_1\cap J\ne\emptyset$ (by the conditions set
out in the statement of the lemma), and $B_1\cup\{j_1\}$ is
closed. Then either we have constructed the desired vertical
3-partition, or there is some $j_2\in A_1\cap J$ such that there
exists a vertical 3-partition $(A_2,j_2,B_2)$ with $A_2\subseteq
A_1$ and $B_2\cup\{j_2\}$ is closed. Since $B_1\cup \{j_1\}$ is
closed, and $j_2\in A_1\cap\cl(B_2)$, it follows that
$\ra(A_2)<\ra(A_1)$. Also $A_2\cap J\ne\emptyset$, by the
conditions set out in the statement of the lemma. We may repeat
this process, each time choosing $j_i\in A_i\cap J$, until we
produce the desired vertical 3-partition $(A_k,h_k,B_k)$. We will
eventually achieve this since $\ra(A_i)<\ra(A_{i-1})$ always.
\end{proof}

We now prove Theorem~\ref{main2}, which tells us that when we have
a 3-connected matroid with a vertical 3-partition, then we can
always find some element on either of the large sides of the
partition, whose contraction keeps us 3-connected up to parallel
classes. We restate Theorem~\ref{main2} here for ease of reading.

\begin{theorem}\label{biglem}
Let $(X_1,x,X_2)$ be a vertical 3-partition of a 3-connected
matroid $M$. Then there exists $y\in X_i$, $i=1,2$, such that
$\si(M/y)$ is 3-connected.
\end{theorem}

\begin{proof}
Suppose that the lemma is false, and suppose that $(X_1,x,X_2)$ is
a vertical 3-partition of $M$, such that for all $y\in X_1$,
$\si(M/y)$ is not 3-connected. Then we may assume by the
construction detailed in the proof of Lemma~\ref{minimal}, that
$X_2\cup\{x\}$ is closed and for all $y\in X_1$, whenever
$(Y_1,y,Y_2)$ is a vertical 3-partition, then $Y_1\cap X_2 \ne
\emptyset$ and $Y_2\cap X_2\ne\emptyset$.
\begin{figure}
\begin{tikzpicture}
    \draw (0,0)--(0,3) (1.2,0)--(1.2,3) (1.8,0)--(1.8,3)
    (3,0)--(3,3);
    \draw (0,0)--(3,0) (0,1.2)--(3,1.2) (0,1.8)--(3,1.8)
    (0,3)--(3,3);
    \filldraw[black] (0.6,1.5)circle(3pt) (1.5,2.4)circle(3pt);
    \draw (0.6,0.6)node{$\ne\emptyset$} (0.6,2.4)node{$\ne\emptyset$}
    (2.4,0.6)node{$\ne\emptyset$} (2.4,2.4)node{$\ne\emptyset$};
    \draw (-0.4,2.4)node{$X_1$} (-0.4,0.6)node{$X_2$} (0.6,3.3)node{$Y_1$}
    (2.4,3.3)node{$Y_2$} (-0.4,1.5)node{$x$} (1.5,3.3)node{$y$};
\end{tikzpicture}
\caption{A Venn diagram showing the 3-partitions $(X_1,x,X_2)$ and
$(Y_1,y,Y_2)$.} \label{venn}
\end{figure}

Let $y\in X_1$ and $(Y_1,y,Y_2)$ be a vertical 3-partition of $M$
with $x\in Y_1$. Consider the Venn diagram for $E(M)$ of Figure
\ref{venn}. By construction, $Y_1\cap X_2\ne\emptyset$ and
$Y_2\cap X_2\ne \emptyset$. Also, since $X_2\cup \{x\}$ is closed,
we see that $y\notin\cl(X_2\cup\{x\})$. However $y\in\cl(Y_1)$ and
$y\in\cl(Y_2)$, meaning that $Y_1\cap X_1\ne\emptyset$ and
$Y_2\cap X_1\ne\emptyset$. We now consider the connectivity of
these sets. We know that $X_2\cup\{x\}$ and $Y_1$ are 3-separating
in $M$ and intersect in at least two elements, so by uncrossing,
$(X_1\cap Y_2)\cup\{y\}$ is 3-separating in $M$. By a similar
argument, $X_1\cap Y_2$ is also 3-separating.

\begin{sublemma}\label{rank}
$\ra((X_1\cap Y_2)\cup\{y\})= 2$.
\end{sublemma}
\begin{proof}
Suppose that $\ra((X_1\cap Y_2)\cup\{y\})\geq 3$. Then
$\lambda((X_1\cap Y_2)\cup\{y\})=\lambda(X_1\cap Y_2)=2$. Also
$y\in\cl(Y_1)$, hence by Lemma \ref{guts}, $y\in\cl(X_1\cap Y_2)$,
so that $\ra(X_1\cap Y_2)\geq 3$. It follows that $(X_2\cup
Y_1,y,X_1\cap Y_2)$ is a vertical 3-partition of $M$ with
$(X_1\cap Y_2)\cap X_2=\emptyset$, a contradiction to our
construction of $(X_1,x,X_2)$.
\end{proof}

\begin{sublemma}\label{subnew1} If $(X_1\cap Y_1)\cup\{x,y\}$ is 3-separating
in $M$ then $\ra((X_1\cap Y_1)\cup\{x,y\})=2$.
\end{sublemma}
\begin{proof}
Suppose $(X_1\cap Y_1)\cup\{x,y\}$ is 3-separating. Then since
$x\in\cl(X_2)$ and $y\in\cl(Y_2)$, it follows that each of
$X_1\cap Y_1$, $(X_1\cap Y_1)\cup\{x\}$, $(X_1\cap Y_1)\cup\{y\}$
and $(X_1\cap Y_1)\cup\{x,y\}$ is 3-separating. By a similar
argument to Sublemma \ref{rank}, we see that $\ra((X_1\cap
Y_1)\cup\{x,y\})=2$.
\end{proof}

\begin{sublemma}\label{subn2}
$(X_1\cap Y_1)\cup\{x,y\}$ is not 3-separating in $M$.
\end{sublemma}

\begin{proof}
Suppose $(X_1\cap Y_1)\cup\{x,y\}$ is 3-separating in $M$. Then by
Sublemmas~\ref{rank} and \ref{subnew1}, $\ra((X_1\cap
Y_2)\cup\{y\})=2$ and $\ra((X_1\cap Y_1)\cup\{x,y\})=2$, hence
$\ra(X_1)=3$. Now suppose that $|X_1\cap Y_2|\geq 2$. Then
$y\in\cl(X_1\cap Y_2)$ by Lemma~\ref{guts}. We now choose $z\in
X_1\cap Y_1$ and consider a vertical 3-partition $(Z_1,z,Z_2)$. We
may assume by symmetry that $Z_1\cap((X_1\cap Y_1)\cup\{x,y\})\ne
\emptyset$ and that $Z_1\cup\{z\}$ is closed by
Lemma~\ref{vertcl}. Then since $z\in\cl(Z_1)$, it follows that
$(X_1\cap Y_1)\cup\{x,y\}\subseteq\cl(Z_1)$, therefore $(X_1\cap
Y_1)\cup\{x,y\}\subseteq Z_1\cup\{z\}$. Observe that $Z_2\cap
X_1\ne\emptyset$ because $z\in\cl(Z_2)$ and
$z\notin\cl(X_2\cup\{x\})$, and as a result, $Z_2\cap (X_1\cap
Y_2)\ne \emptyset$. Furthermore, as $Z_1\cup\{z\}$ is closed and
$y\in Z_1$, it follows that $Z_1\cap(X_1\cap Y_2)=\emptyset$,
otherwise we would have the contradiction that $(X_1\cap
Y_2)\subseteq \cl(Z_1)$. This means that $(X_1\cap Y_2)\subseteq
Z_2$, resulting in $\{y,z\}\subseteq\cl(Z_2)$. Furthermore, this
implies that $(X_1\cap Y_1)\cup\{y\}\subseteq\cl(Z_2)$, because
$(X_1\cap Y_1)\cup\{y\}\subseteq\cl(\{y,z\})$. By
Lemma~\ref{vertcl}, we may now construct the vertical 3-partition
$(Z_1-\cl(Z_2),z,\cl(Z_2)-\{z\})$ which has $X_1\subseteq
\cl(Z_2)$. This is a contradiction since that would mean that
$z\in\cl(Z_1-\cl(Z_2))$, which is impossible as
$Z_1-\cl(Z_2)\subseteq X_2\cup\{x\}$. This contradiction shows
that if $(X_1\cap Y_1)\cup\{x,y\}$ is 3-separating, then we cannot
have $|X_1\cap Y_2|\geq 2$, and we see that $|X_1\cap Y_2|=1$.

Letting $w\in X_1\cap Y_2$, it follows easily that $((X_1\cap
Y_1)\cup\{x,y\}, X_2)$ is a vertical 2-separation of $M\ba w$. By
Bixby's Theorem~\ref{bixby}, it follows that $\si(M/w)$ is
3-connected, contradicting our original assumption that for all
$e\in X_1$, $\si(M/e)$ is not 3-connected. The result follows.
\end{proof}

Consider the size of $X_2\cap Y_2$. If $|X_2\cap Y_2|\geq 2$, then
by uncrossing, $(X_1\cap Y_1)\cup\{x,y\}$ is 3-separating,
contradicting Sublemma~\ref{subn2}. Hence, it must be the case
that $|X_2\cap Y_2|=1$ and $\lambda((X_1\cap Y_1)\cup\{x,y\})=3$
(by uncrossing, we must have $\lambda((X_1\cap
Y_1)\cup\{x,y\})\leq 3$).

Proceeding from here, it is helpful to continue to refer to the
Venn diagram of Figure~\ref{venn} to gain intuition. We have
$|Y_2|\geq 3$ and $|Y_2\cap X_2|=1$, hence $|X_1\cap Y_2|\geq 2$
meaning that $(X_1\cap Y_2)\cup\{y\}$ is a segment of size at
least three. It is clear also that since $|X_2\cap Y_2|=1$,
$\ra(X_1\cap Y_2)=2$, and $\ra(Y_2)\geq 3$, it follows that
$Y_1\cup\{y\}$ is closed and $\ra(Y_2)=3$. Evidently, no member of
$X_1\cap Y_1$ can extend $(X_1\cap Y_2)\cup\{y\}$ to a larger
segment because $Y_1\cup\{y\}$ is closed. Hence $(X_1\cap
Y_2)\cup\{y\}$ is a maximal segment contained in $X_1$.

\begin{sublemma}\label{rank2}
$y\in\cl\, ((X_1\cap Y_1)\cup\{x\})$.
\end{sublemma}
\begin{proof}
Suppose that $y\notin\cl((X_1\cap Y_1)\cup\{x\})$. Then since
$y\in\cl(Y_2)$, we have $\lambda((X_1\cap
Y_1)\cup\{x\})=\lambda((X_1\cap Y_1)\cup\{x,y\})-1=2$.

We now see that $((X_1\cap Y_1)\cup\{x\},X_2,(X_1\cap
Y_2)\cup\{y\})$ is an exact 3-partition of $M$ with
$\sqcap((X_1\cap Y_1)\cup\{x\},X_2)\geq 1$ since $x\in\cl(X_2)$.
By Lemma~\ref{flowers1}, $\sqcap((X_1\cap Y_1)\cup\{x\},(X_1\cap
Y_2)\cup \{y\})\geq 1$. We also see that since $(X_1\cap
Y_2)\cup\{y\}$ is a segment and $X_1\cap
Y_2\not\subseteq\cl(Y_1)$, we have $\sqcap((X_1\cap Y_2)\cup\{y\},
Y_1)=1$. By Lemma \ref{flowers1}, $\sqcap((X_1\cap
Y_1)\cup\{x\},(X_1\cap Y_2)\cup\{y\})=1$, and by Lemma
\ref{flowers2}, we have $y\in\cl((X_1\cap Y_1)\cup\{x\})$ because
$y\in\cl(Y_1)$. This contradicts our initial assumption, and we
conclude that $y\in\cl((X_1\cap Y_1)\cup\{x\})$.
\end{proof}

Let $s\in X_1\cap Y_2$, and consider a vertical 3-partition
$(S_1,s,S_2)$ with $x\in S_1$. By the symmetry of the situation,
$(S_1,s,S_2)$ shares many of the same properties as $(Y_1,y,Y_2)$,
for example $|X_2\cap S_2|=1$, $(X_1\cap S_2)\cup\{s\}$ is a
maximal segment contained in $X_1$, also $s\in\cl((X_1\cap
S_1)\cup\{x\})$ and $S_1\cup \{s\}$ is closed with $\ra(S_2)=3$.
Consider the members of the segment $(X_1\cap S_2)\cup\{s\}$.
Since $s\in\cl(X_1\cap S_2)$ and $s\notin\cl((X_1\cap
Y_1)\cup\{y\})$ (recall that $Y_1\cup\{y\}$ is closed), there must
be some member $s'$ of $X_1\cap Y_2$ that is contained in $X_1\cap
S_2$. Now, as $\{s,s'\}$ is a subset of $(X_1\cap Y_2)\cup\{y\}$
and $(X_1\cap S_2)\cup\{s\}$, both of which are maximal segments
contained in $X_1$, it follows that $(X_1\cap
Y_2)\cup\{y\}=(X_1\cap S_2)\{s\}$. This implies that $X_1\cap
Y_1=X_1\cap S_1$, and we see that $\{y,s\}\subseteq\cl((X_1\cap
Y_1)\cup\{x\})$, a contradiction as we have already established
that $Y_1\cup\{y\}$ is closed and $s\notin Y_1\cup\{y\}$.

We conclude from this final contradiction that our original
assumption, that for all $e\in X_1$, $\si(M/e)$ is not
3-connected, must be false. The result now follows by a symmetric
argument on $X_2$.
\end{proof}

The result of Theorem~\ref{main2} can now be put to use on our
problem, and we obtain the following corollary.

\begin{corollary}\label{bigcor}
Let $M$ be a 3-connected matroid with a hyperplane $H$, such that
for all $h\in H$, $\si(M/h)$ is not 3-connected. Let $(X_1,x,X_2)$
be a vertical 3-partition of $M$ with $x\in H$, and let $C$ be the
cocircuit whose complement is $H$. Then $X_i\cap H\ne\emptyset$
and $X_i\cap C\ne \emptyset$, for $i=1,2$.
\end{corollary}

\begin{proof}
Suppose first that $(X_1,x,X_2)$ is a vertical 3-partition of $M$
with $x\in H$. Then by Theorem~\ref{main2}, there exists $y\in
X_i$, $i=1,2$, such that $\si(M/y)$ is 3-connected. Since
$\si(M/h)$ is not 3-connected for all $h\in H$, we see that
$X_i\cap C\ne\emptyset$, $i=1,2$.

Now suppose that $X_1\cap H=\emptyset$, so that $X_1\subseteq C$
and $H\subseteq X_2\cup\{x\}$. Since $(X_1,x,X_2)$ is a vertical
3-partition, $\ra(X_2\cup\{x\})<\ra(M)$ meaning that
$X_2\cup\{x\}$ is contained in some hyperplane $H'$ of $M$. Thus
$H\subseteq X_2\cup\{x\}\subseteq H'$, implying that
$H=X_2\cup\{x\}$, contradicting Corollary~\ref{bigcor} which
states that $X_2\cap C\ne\emptyset$. The result now follows by a
symmetric argument on $X_2\cap H$.
\end{proof}

The following lemma will be used in the proof of Lemma~\ref{dual},
which is an important part of the proof of our main theorem.
\begin{lemma}\label{coseg2}
Let $M$ be a 3-connected matroid, and let $X$ and $Y$ be disjoint
subsets of $E(M)$, where $X$ is a cosegment. If for some $x\in X$,
$\sqcap(X-\{x\},Y)\geq 1$, then $X$ is a maximal member of the
class of all cosegments of $M$ that do not intersect $Y$.
\end{lemma}
\begin{proof}
Suppose this is false, and that for some $e\in E(M)-(X\cup Y)$,
$X\cup\{e\}$ is a cosegment of $M$. Then since $e$ and $x$ are
distinct members of $X\cup\{e\}$, the remaining members of
$X\cup\{e\}$ become coloops in the matroid $M\ba \{e,x\}$. It
follows that $\sqcap(X-\{x\},E(M)-(X\cup\{e\}))=0$, implying that
$\sqcap(X-\{x\},Y)=0$ by Lemma~\ref{flowers2-}. The result follows
by contradiction.
\end{proof}

For the next lemma, we define \textit{covertical $k$-partitions}
and \textit{covertical $k$-separations} of a matroid to be
vertical $k$-partitions and $k$-separations of the dual matroid
respectively. For this lemma, we consider the dual of our problem,
namely that our 3-connected matroid has a cohyperplane from which
deletion of any element leaves the matroid with a covertical
2-separation. Here, we consider only the case where the complement
of the cohyperplane is a triangle.

\begin{lemma}\label{dual}
Let $M$ be a 3-connected matroid with a cohyperplane $H$, such
that for all $h\in H$, $\co(M\ba h)$ is not 3-connected, and let
$C$ be the circuit whose complement is $H$. Suppose that $C$ is a
triangle of $M$. Then $M$ is a member of the family $\mathcal{P}$
of matroids defined in Section 1.
\end{lemma}
\begin{proof}
Firstly note that in $M^*$, $H$ is a hyperplane such that for all
$h\in H$, $\si(M/h)$ is not 3-connected. We may assume by
Lemma~\ref{size} that $|E(M)|\geq 7$. Let $C=\{c_1,c_2,c_3\}$, and
let $x\in H$ with $(X_1,x,X_2)$ a covertical 3-partition of $M$.
Then by Corollary~\ref{bigcor}, $C\cap X_i\ne\emptyset$, $i=1,2$.
We may assume without loss of generality that $C\cap X_1=\{c_1\}$,
giving $c_1\in\cl(X_2)$ (because $\{c_1,c_2,c_3\}$ is a triangle).
This implies that $(X_1-\{c_1\},X_2\cup\{c_1\})$ is a 2-separation
of $M\ba x$, however it is not a covertical 2-separation since
$(X_1-\{c_1\})\cap C=\emptyset$ which would contradict
Corollary~\ref{bigcor}. Since $(X_1-\{c_1\},X_2\cup\{c_1\})$ is
not covertical, $X_1-\{c_1\}$ is a series class of $M\ba x$. It
follows that $X_1-\{c_1\}\cup \{x\}$ is a cosegment of $M$. We
also see that $c_1\in\cl_M(X_1-\{c_1\})$, because otherwise
$X_1-\{c_1\}$ would be a separator of $M\ba x$, a contradiction to
the connectivity of $M$. Thus by Lemma \ref{circuit}, $X_1$ is a
circuit of $M\ba x$, and hence a circuit of $M$.

Now, let $y\in X_1-\{c_1\}$ and let $(Y_1,y,Y_2)$ be a covertical
3-partition of $M$, where $Y_1\cap C=\{c_i\}$. Again,
$Y_1-\{c_i\}\cup\{y\}$ is a cosegment of $M$. Let $\{y,z,w\}$ be a
triad of this cosegment that contains $y$. By orthogonality,
$\{z,w\}\cap (X_1-\{c_1\})\ne\emptyset$ since $X_1$ is a circuit
containing $y$, and $c_1\notin\{z,w\}$. It now follows that
$\{y,z,w\}$ intersects a triad of $X_1-\{c_1\}\cup\{x\}$ in at
least two members, so that $X_1-\{c_1\}\cup\{x,y,z,w\}$ is a
cosegment of $M$. Now, observe that $X_1-\{c_1\}\cup\{x\}$ is a
maximal cosegment of $E(M)-C$ by Lemma \ref{coseg2}, because
$\sqcap(X_1-\{c_1\},C)= 1$. It now follows that
$X_1-\{c_1\}\cup\{x,y,z,w\}= X_1-\{c_1\}\cup\{x\}$, and we deduce
that $Y_1-\{c_i\}\subseteq X_1\cup\{x\}$. A symmetric argument now
shows that $Y_1-\{c_i\}\cup\{y\}=X_1-\{c_1\}\cup\{x\}$. Now,
observe that since $X_1$ is a circuit of $M$,
$y\in\cl(X_1-\{y\})$, but also $y\notin\cl(Y_1)$, implying that
$c_i\ne c_1$. Thus we have without loss of generality that
$c_i=c_2$, that is $\{c_2\}= Y_1\cap C$.

We now consider $z\in X_1-\{c_1,y\}$ and a covertical 3-partition
$(Z_1,z,Z_2)$ of $M$, with $|Z_1\cap C|=1$. Then by the symmetry
of the argument above, $c_3\in Z_1$ and
$Z_1-\{c_3\}\cup\{z\}=X_1-\{c_1\}\cup\{x\}$ is a cosegment of $M$.
Now suppose that there is another member $w\in X_1-\{c_1,y,z\}$.
Then if $(W_1,w,W_2)$ were a covertical 3-partition of $M$ with
$|W_1\cap C|=1$, a symmetrical argument tells us that
$c_1,c_2,c_3\notin W_1$, a contradiction. We conclude that no such
$w$ exists, and that $X_1=\{c_1,y,z\}$. The result of this is that
$\{x,y,z\}$ is a maximal cosegment of $M$, and since $X_1$, $Y_1$
and $Z_1$ are circuits, we have $M|_{\{x,y,z,c_1,c_2,c_3\}}\cong
K_4$.

We may apply the argument above to any member $h\in H$, to show
that $h$ is contained in a triad $T$, and that $M|_{T\cup C}\cong
K_4$. We conclude that $M\in\mathcal{P}$.
\end{proof}

\section{Proof of main theorem.}
In this section, we complete the proof of the main theorem of the
paper, Theorem~\ref{main}. We restate Theorem~\ref{main} here for
ease of reading.

\begin{theorem}[Main theorem]\label{mainth}
A 3-connected matroid $M$ has a hyperplane $H$, such that for all
$h\in H$, $\si(M/h)$ is not 3-connected if and only if $M$ is a
member of the class of matroids $\mathcal{P}^*$ defined in
Section~1.
\end{theorem}
\begin{proof}
It is easily seen that all members of $\mathcal{P}^*$ have such a
hyperplane $H$, if we let $H$ be the set of elements
$\{t_{11},t_{12},t_{13},\dots,t_{n1},t_{n2},t_{n3}\}$.

We must now show that if $M$ has a hyperplane $H$ with the
contraction property stated above, then $M\in\mathcal{P}^*$. Let
$C$ be the cocircuit whose complement is $H$. Let $x\in H$ and let
$(X_1,x,X_2)$ be a vertical 3-partition of $M$ such that
$X_2\cup\{x\}$ is closed, and for all $y\in X_1\cap H$, whenever
$(Y_1,y,Y_2)$ is a vertical 3-partition of $M$, then $Y_1\cap
X_2\ne \emptyset$ and $Y_2\cap X_2\ne \emptyset$ (we know that
such a 3-partition exists by Lemma~\ref{minimal} and
Corollary~\ref{bigcor}). Corollary~\ref{bigcor} now tells us that
$X_i\cap C\ne\emptyset$ and $X_i\cap H\ne \emptyset$, for $i=1,2$.

Let $y\in X_1\cap H$ and $(Y_1,y,Y_2)$ be a vertical 3-partition
of $M$. Then as in the proof of Theorem~\ref{biglem}, each of
$X_1\cap Y_1$, $X_1\cap Y_2$, $X_2\cap Y_1$ and $X_2\cap Y_2$ is
nonempty. Now, observe by uncrossing that $X_1\cap Y_2$ and
$(X_1\cap Y_2)\cup\{y\}$ are 3-separators of $M$. If $|X_1\cap
Y_2|\geq 2$ then $\ra((X_1\cap Y_2)\cup\{y\})=2$ by a proof
similar to that of Sublemma~\ref{rank}.

\begin{sublemma}\label{sublem1}
If $|X_2\cap Y_2|\geq 2$, then $(X_1\cap Y_1)\cup\{x,y\}$ is a
segment contained in $H$, and $X_1\cap Y_2\subseteq C$.
\end{sublemma}

\begin{proof}
Suppose that $|X_2\cap Y_2|\geq 2$. Then by uncrossing, $(X_1\cap
Y_1)\cup\{x,y\}$, $(X_1\cap Y_1)\cup\{y\}$, $(X_1\cap
Y_1)\cup\{x\}$ and $X_1\cap Y_1$ are 3-separators of $M$. Since
$y\in\cl(Y_2)$, we have $y\in\cl((X_1\cap Y_1)\cup\{x\})$ by
Lemma~\ref{guts}, and a similar argument gives $x\in\cl((X_1\cap
Y_1)\cup\{y\})$. Also, if $|X_1\cap Y_1|\geq 2$ then
$x,y\in\cl(X_1\cap Y_1)$. We see that $\ra((X_1\cap
Y_1)\cup\{x,y\})=2$ by a similar proof to that of
Sublemma~\ref{rank}. Now, since $x,y\in H$, and $H$ is closed, we
see that $X_1\cap Y_1\subseteq H$. Furthermore, $(X_1\cap Y_2)\cap
C\ne\emptyset$ because we have $X_1\cap C\ne\emptyset$ and
$((X_1\cap Y_1)\cup\{y\})\cap C=\emptyset$. We now see that
$X_1\cap Y_2 \subseteq C$, because $H$ is closed, $y\in H$, and
$\ra((X_1\cap Y_2)\cup\{y\})=2$.
\end{proof}

\begin{sublemma}\label{sublem2}
If $|X_2\cap Y_2|\geq 2$ then $|X_1\cap Y_1|=1$.
\end{sublemma}

\begin{proof}
Suppose that $|X_2\cap Y_2|\geq 2$ and $|X_1\cap Y_1|\geq 2$. Then
by Corollary~\ref{segment}, we must have $|X_1\cap Y_2|\geq 2$ as
well. Choose $z\in X_1\cap Y_1$ and consider a vertical
3-partition $(Z_1,z,Z_2)$ of $M$. We may assume without loss of
generality that $x\in Z_1$. Combine this with the fact that $z\in
\cl(Z_1)$, to obtain $(X_1\cap Y_1)\cup\{x,y\}\subseteq\cl(Z_1)$.
Hence we may assume by Lemma~\ref{vertcl} that $(X_1\cap
Y_1)\cup\{x,y\}\subseteq Z_1$. Now, $z\in\cl(Z_2)$ and
$z\notin\cl(X_2)$ so we see that $(X_1\cap Y_2)\cap
Z_2\ne\emptyset$. Suppose that $X_1\cap Y_2\not\subseteq Z_2$ and
that $w\in(X_1\cap Y_2)\cap Z_1$. Then $\ra(Z_1\cap X_1)=3$, which
implies that $X_1\subseteq\cl(Z_1)$. Then by Lemma~\ref{vertcl},
$(\cl(Z_1)-\{z\},z,Z_2-\cl(Z_1))$ is a vertical 3-partition of $M$
with $Z_2-\cl(Z_1)\subseteq X_2$, contradicting the fact that
$z\notin\cl(X_2)$. We conclude that $X_1\cap Y_2\subseteq Z_2$.
Now, since $y\in\cl(X_1\cap Y_2)$, we see that
$\{y,z\}\subseteq\cl(Z_2)$ implying that $(X_1\cap
Y_1)\cup\{x,y\}\subseteq\cl(Z_2)$. By Lemma~\ref{vertcl}, we may
construct the new 3-partition $(Z_1-\cl(Z_2),z,\cl(Z_2)-\{z\})$ in
which $Z_1-\cl(Z_2)\subseteq X_2$, contradicting that
$z\notin\cl(X_2)$. We may conclude from this that it is not
possible to have $|X_1\cap Y_1|\geq 2$, and the result follows.
\end{proof}

\begin{sublemma}\label{sublem3}
If $|X_2\cap Y_2|\geq 2$, then $M$ is a member of the class
$\mathcal{P}^*$.
\end{sublemma}

\begin{proof}
Suppose that $|X_2\cap Y_2|\geq 2$. Then by
Sublemma~\ref{sublem2}, $|X_1\cap Y_1|=1$. Corollary~\ref{segment}
tells us that $|X_1\cap Y_2|=1$. Thus $X_1$ is a triad of $M$. Let
$z\in X_1\cap Y_1$, and let $a\in X_1\cap Y_2$. Then $z\in H$ as
$H$ is closed, and by Corollary~\ref{bigcor}, $a\in C$. We also
see by uncrossing that $(X_2\cap Y_1)\cup\{x\}$ is a 3-separator
of $M$, and hence a 2-separator  of $M\ba z$, which by Bixby's
Theorem \ref{bixby} implies that $|X_2\cap Y_1|=1$. We thus have
$Y_1$ a triad of $M$. Let $b\in X_2\cap Y_1$. Then by
Corollary~\ref{bigcor}, $Y_1\cap C\ne \emptyset$ implying that
$b\in C$.

Consider a vertical 3-partition $(Z_1,z,Z_2)$ of $M$, and assume
without loss of generality that $x\in Z_1$. Then since
$z\in\cl(Z_1)$, we have $y\in\cl(Z_1)$, so we may assume by Lemma
\ref{vertcl}, that $y\in Z_1$.

Now, the triads $\{a,y,z\}$ and $\{b,x,z\}$ must both intersect
$Z_2$ in order for $z\in\cl(Z_2)$, thus $a,b\in Z_2$. We have
$\lambda(\{x,y,z\})=2$, and $\ra(X_2\cap Y_2)=\ra((X_2\cap
Y_2)\cup\{a,b\})-2$, thus $\sqcap(\{x,y,z\},X_2\cap Y_2)=0$,
giving
$\ra(Z_1)=\ra(\{x,y\})+\ra(Z_1-\{x,y\})=\ra(Z_1-\{x,y\})+2$. Now,
since $z\in\cl(Z_2)$, $\ra(Z_2\cup\{x,y\})\leq\ra(Z_2)+1$. It
follows that $\lambda(Z_1-\{x,y\})\leq 1$ which implies that
$|Z_1-\{x,y\}|=1$. Then $Z_1$ is a triad of $M$, $Z_1\cup\{z\}$ is
a fan of $M$, and letting $c\in Z_1-\{x,y\}$, $M\ba c$ has the
vertical 2-separation $(\{x,y,z\},Z_2)$, thus $c\in C$. We may now
deduce that $a,b,c\in\cl^*(\{x,y,z\})$, and by
Lemma~\ref{cosegment}, $\{a,b,c\}$ is a triad of $M$ contained in
$C$. As $C$ is a cocircuit, we have $C=\{a,b,c\}$. We also see by
the list of triads and triangles in $\{x,y,z,a,b,c\}$, that
$M^*|_{\{x,y,z,a,b,c\}}\cong K_4$. We may now apply Lemma
\ref{dual} in order to obtain the result that $M^*$ is a member of
$\mathcal{P}$.
\end{proof}

Having considered the case where $|X_2\cap Y_2|\geq 2$, we must
now look at the case where $|X_2\cap Y_2|=1$. Firstly, $|Y_2|\geq
3$ giving $|X_1\cap Y_2|\geq 2$, hence $(X_1\cap Y_2)\cup \{y\}$
is a segment of size at least three. Let $e\in X_2\cap Y_2$. Then
$((X_1\cap Y_2)\cup\{y\},Y_1)$ is a vertical 2-separation of $M\ba
e$, implying that $e\in C$ by Bixby's Theorem~\ref{bixby}. Note
that $X_1\cap Y_2$ contains an element of $H$, by
Corollary~\ref{bigcor} applied to $(Y_1,y,Y_2)$. Since $H$ is
closed and $y\in H$, it follows that $(X_1\cap
Y_2)\cup\{y\}\subseteq H$. We may now apply
Corollary~\ref{segment} to the 3-separation $(Y_1,Y_2\cup\{y\})$
to obtain $|X_1\cap Y_2|=2$.

\begin{sublemma}\label{sublem4}
$y\in\cl\, ((X_1\cap Y_1)\cup\{x\})$.
\end{sublemma}
\begin{proof}
This is identical to the proof of Sublemma~\ref{rank2}.
\end{proof}

Let $z\in X_1\cap Y_2$ and consider a vertical 3-partition
$(Z_1,z,Z_2)$ of $M$ with $x\in Z_1$. We know by
Sublemma~\ref{sublem3} that if $|Z_2\cap X_2|\geq 2$, then $M$ is
in the class $\mathcal{P}^*$, so we may assume that $|Z_2\cap
X_2|=1$, and by symmetry, we see that $Z_2$ is a triad of $M$ with
$(X_1\cap Z_2)\cup\{z\}$ a triangle contained in $H$.

Let $w$ be the third member of the triangle $(X_1\cap
Y_2)\cup\{y\}$. Then since $Z_1\cup\{z\}$ is closed, either
$\{y,w\}\subseteq Z_1$ or $\{y,w\}\subseteq Z_2$. If
$\{y,w\}\subseteq Z_1$ then $(X_1\cap Z_2)\cup\{z\}$ is a triangle
with $X_1\cap Z_2\subseteq X_1\cap Y_1$, but this is not possible
because $z$ is not in $\cl(Y_1)$. Therefore, we must have
$\{y,w\}\subseteq Z_2$, and by the sizes of $Y_2$ and $Z_2$, we
have $\{y,w\}=X_1\cap Z_2$, which implies that $X_1\cap
Y_1=X_1\cap Z_1$. However, $y\in\cl((X_1\cap Y_1)\cup\{x\})$ and
hence $y\in\cl(Z_1)$, contradicting that $Z_1\cup\{z\}$ is closed.
This contradiction completes the analysis of the case where
$|X_2\cap Y_2|=1$, and the result of Theorem~\ref{mainth} (ie.
Theorem~\ref{main}) now follows.
\end{proof}

\end{document}